\documentclass[11pt, a4paper]{article}
\usepackage{amsmath,amssymb,latexsym,graphics,epsfig,url}
\usepackage{color}

\usepackage{amssymb}
\usepackage{amsthm}
\usepackage[english]{babel}
\usepackage{epsfig}
\usepackage{graphicx}
\newtheorem{p-lema}{The Parity Lemma}
\newtheorem{theo}{Theorem}[section]
\newtheorem{propo}[theo]{Proposition}
\newtheorem{lema}[theo]{Lemma}

\input amssym.def
\newsymbol\rtimes 226F
\newfont{\nset}{msbm10}
\newcommand{\ns}[1]{\mbox{\nset #1}}

\def\Col{\mbox{{\cal Col}}}

\def\Ei{{\cal E}}

\def\R{\ns{R}}

\def\matrix0{{\mbox {\boldmath $O$}}}

\def\e{{\mbox{\boldmath $e$}}}
\def\ep{\epsilon}

\def\vv{{\mbox{\boldmath $v$}}}

\def\vec0{\mbox{\bf 0}}

\def\dist{\mathop{\partial }\nolimits}

\def\mod{\mathop{\rm mod }\nolimits}

\begin{document}

\title{Some Results on the Structure of \\  Multipoles in the Study of Snarks
\thanks{Research supported by the Ministerio de Educaci\'on y
Ciencia, Spain, and the European Regional Development Fund under
project MTM2011-28800-C02-01 and by the Catalan Research Council
under project 2005SGR00256.}}

\author{M.A. Fiol, J. Vilaltella
\\ \\
{\small Universitat Polit\`ecnica de Catalunya, BarcelonaTech} \\
{\small Dept. de Matem\`atica Aplicada IV, Barcelona, Catalonia}\\
{\small (e-mails: {\tt
\{fiol,joan.vilaltella\}@ma4.upc.edu})} \\
 }
\maketitle

\noindent {\em Keywords:} Cubic graph, edge-coloring, snark, multipole, Parity Lemma, states, color complete, color closed, separable, irreducible, tree, cycle, linear recurrence.

\noindent {\em AMS classification:} 05C15, 05C05, 05C38.
\begin{abstract}
 Multipoles are the pieces we obtain by cutting some edges of a cubic graph. As a result of the cut, a multipole $M$ has dangling edges with one free end, which we call semiedges. Then, every 3-edge-coloring of a multipole induces a coloring or state of its semiedges, which satisfies the Parity Lemma. Multipoles have been extensively used in the study of snarks, that is, cubic graphs which are not 3-edge-colorable. Some results on the states and structure of the so-called color complete and color closed multipoles are presented. In particular, we give lower and upper linear bounds on the minimum order of a color complete multipole, and compute its exact number of states. Given two multipoles $M_1$ and $M_2$ with the same number of semiedges, we say that $M_1$ is reducible to $M_2$ if the state set of $M_2$ is a non-empty subset of the state set of $M_1$ and $M_2$ has less vertices than $M_1$. The function $v(m)$ is defined as the maximum number of vertices of an irreducible multipole with $m$ semiedges. The exact values of  $v(m)$ are only known for $m\le 5$. We prove that tree and cycle multipoles are irreducible and, as a byproduct, that $v(m)$ has a linear lower bound.
\end{abstract}

\section{Introduction}

A cubic graph can be subdivided by an edge cut into parts called multipoles. As a result of the cut, a multipole has half-edges with one free end. An isolated edge with both free ends can be viewed as a multipole, too. A multipole need not be connected. We call \mbox{$m$-pole} a multipole with $m$ free ends. According to the Parity Lemma, if a multipole can be 3-edge-colored then the number of free ends of each color must have the same parity as $m$. Therefore, the colorings of the free ends are not arbitrary. We refer to such colorings as {\em states}. Multipoles and their states are used to study non-3-edge-colorable cubic graphs, called {\em snarks}. See, for instance, Isaacs \cite{i75}, Gardner \cite{g76}, Goldberg \cite{go81}, and Fiol \cite{f91}. A multipole with all states permitted by the Parity Lemma is known as a {\em color complete} multipole, and a multipole that has at least one state in common with any other multipole is known as a {\em color closed} multipole. Nedela and Skoviera \cite{ns96} posed the question whether non-color complete color closed multipoles exist for $m\ge 5$. Karab\'a\v{s}, M\'ac\v{a}jov\'a and Nedela \cite{kmn13} afirmatively answered this question, giving examples obtained by computer search. We also give examples for $m\in \{5,6\}$ and explain a method to obtain them. Moreover, we give other results on the states and structure of multipoles, in particular of color complete multipoles. Namely, we give an exact formula for the number of states of a color complete $m$-pole, we prove that color complete $m$-poles exist for every non-trivial value of $m$, and we give lower and upper bounds on the minimum order of a color complete $m$-pole as a function of $m$. Given two multipoles $M_1$ and $M_2$ with the same number of semiedges, we say that $M_1$ is reducible to $M_2$ if the state set of $M_2$ is a non-empty subset of the state set of $M_1$ and $M_2$ has less vertices than $M_1$. The function $v(m)$, introduced by Fiol \cite{f91}, is defined as the maximum number of vertices of an irreducible $m$-pole. Only a few values of $v(m)$ are known. We prove that tree and cycle multipoles are irreducible and, as a byproduct, that $v(m)$ has a linear lower bound.

\section{Multipoles and Tait colorings}
A {\em multipole} or {\em $m$-pole}  $M=(V,E,{\cal E})$ consists of a finite set of vertices $V=V(M)$, a set of edges $E=E(M)$ which are unordered pairs of vertices, and a set ${\cal E}={\cal E}(M)$ whose $m$ elements $\epsilon_1,\ldots,\epsilon_m$ are called {\em semiedges}.
Each semiedge is associated either with one vertex or with another semiedge making up what is usually called an {\em isolated} or {\em free} edge. For instance, Fig. \ref{7-pole}$(a)$ shows a $7$-pole with two free edges. Notice that a multipole can be disconnected or even be `empty', in the sense that it can have no vertices. The diagram of a generic $m$-pole will be as shown in Fig. \ref{7-pole}$(b)$.

\begin{figure}[h]
\begin{center}
\includegraphics[width=12cm]{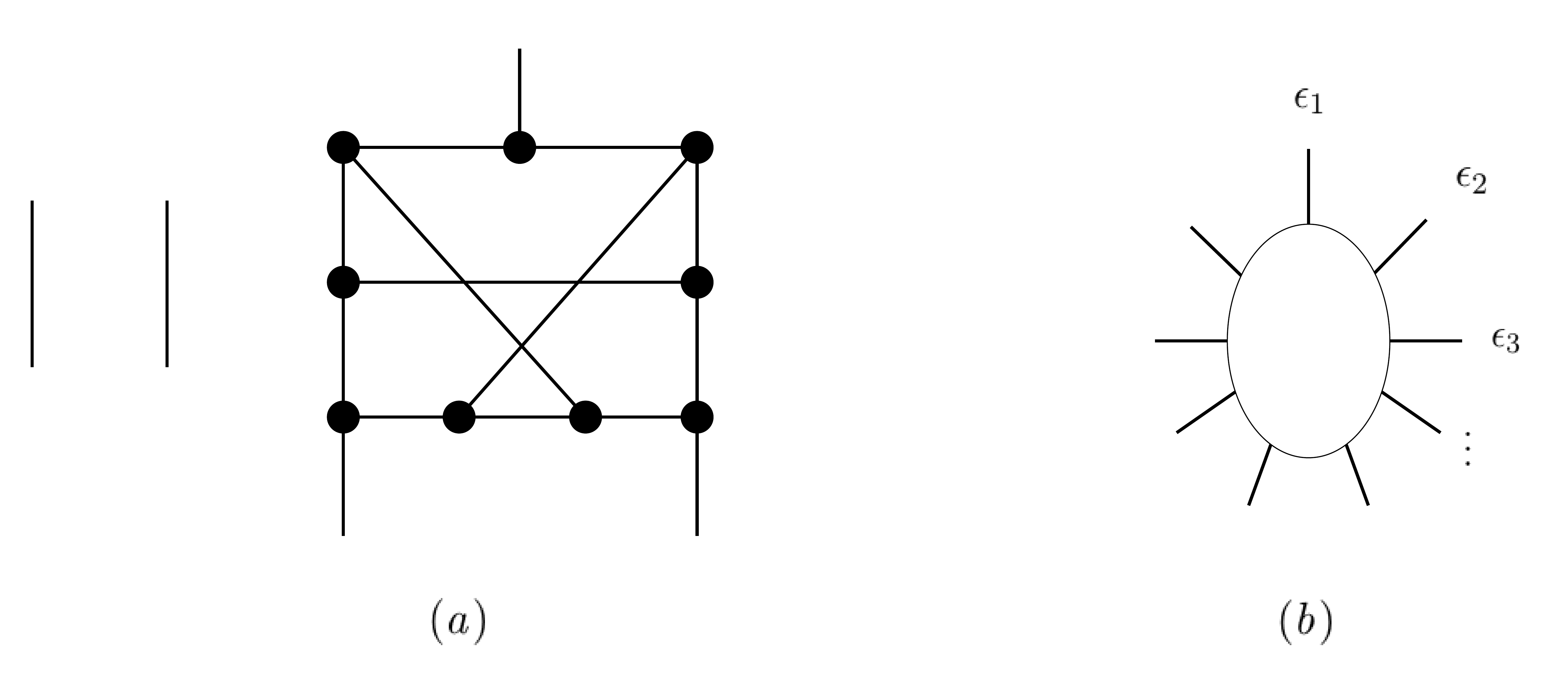}
\caption{A 7-pole with two free edges $(a)$ and a generic multipole $(b)$}
\label{7-pole}
\end{center}
\end{figure}

In this paper we follow notation in \cite{f91,fv12}. For instance, if the semiedge $\epsilon$ is associated with vertex $u$, we say that $\epsilon$ is {\em incident} to $u$ and, following Goldberg's notation \cite{go81}, we write $\epsilon=(u)$. Also, by joining the semiedges $(u)$ and $(v)$ we get the edge $(u,v)$. The {\em degree} of a vertex $u$, denoted by $\delta(u)$, is the number of edges plus the number of semiedges incident to it. In this paper, we only consider cubic multipoles, that is, those for which $\delta(u)=3$ for all $u\in V$. Thus, the simplest $2$- and $3$-poles are a free edge and a vertex with three incident semiedges. We denote them by $\e$ and $\vv$, respectively.

Given a multipole $M$, we denote by $G[M]$ the graph obtained from $M$ by removing all its semiedges. Using this definition, we say that a multipole
$N$ is {\em contained} in $M$, or that $N$ is a {\em submultipole} of $M$, when $G[N]$ is a subgraph of $G[M]$.
Note that, in this case, $N$ can be obtained from $M$ by `cutting' (in one or more points) some of its edges.
The {\em distance} $\dist(\ep,\zeta)$ between two semiedges $\ep=(u)$ and $\zeta=(v)$ is defined to be the distance in $G[M]$ between its incident vertices $u$ and $v$.

When $G[M]$ is a tree or a forest we refer to $M$ as a {\em tree multipole} or a {\em forest multipole} respectively. Analogously, if $G[M]$ is a cycle we say that $M$ is a {\em cycle multipole}.
A simple counting reasoning allows us to see that a tree $m$-pole has $n=m-2$ vertices, whereas a cycle $m$-pole has $n=m$ vertices.
In general,  for every $m$-pole with $n$ vertices, we have
that $m\equiv n\ \mod 2$ and the following result holds:
\begin{lema}
\label{lema-forests}
Let $M$ be an $m$-pole with  $n=|M|$ vertices and $c$ components. Then,
$$
n\ge m-2c,
$$
with equality if and only if $M$ is a forest multipole.
\end{lema}
\begin{proof}
For $i=1,\ldots,c$, let $n^{(i)}$ and $m^{(i)}$ denote, respectively, the numbers of vertices and semiedges of the $i$-th component.
Then the result follows from the fact that $n^{(i)}\ge m^{(i)}-2$ with equality if and only if the $i$-th component is a tree multipole.
\end{proof}

Let $C=\{1,2,3\}$ be a set of `colors'. Then, a {\em 3-edge-coloring}  or {\em Tait coloring} of an $m$-pole $(V,E,\Ei)$ is a mapping $\phi: E\cup {\cal E}\rightarrow C=\{1,2,3\}$ such that all the  edges and/or semiedges incident with a vertex receive distinct colors, and each isolated edge has both semiedges with the same color. For example, Fig. \ref{colored-6-pole} shows a Tait coloring of a 6-pole. Note that the numbers of semiedges with the same color have the same parity. The following basic lemma states that this is always the case.

\begin{figure}[h]
\begin{center}
\includegraphics[width=6cm]{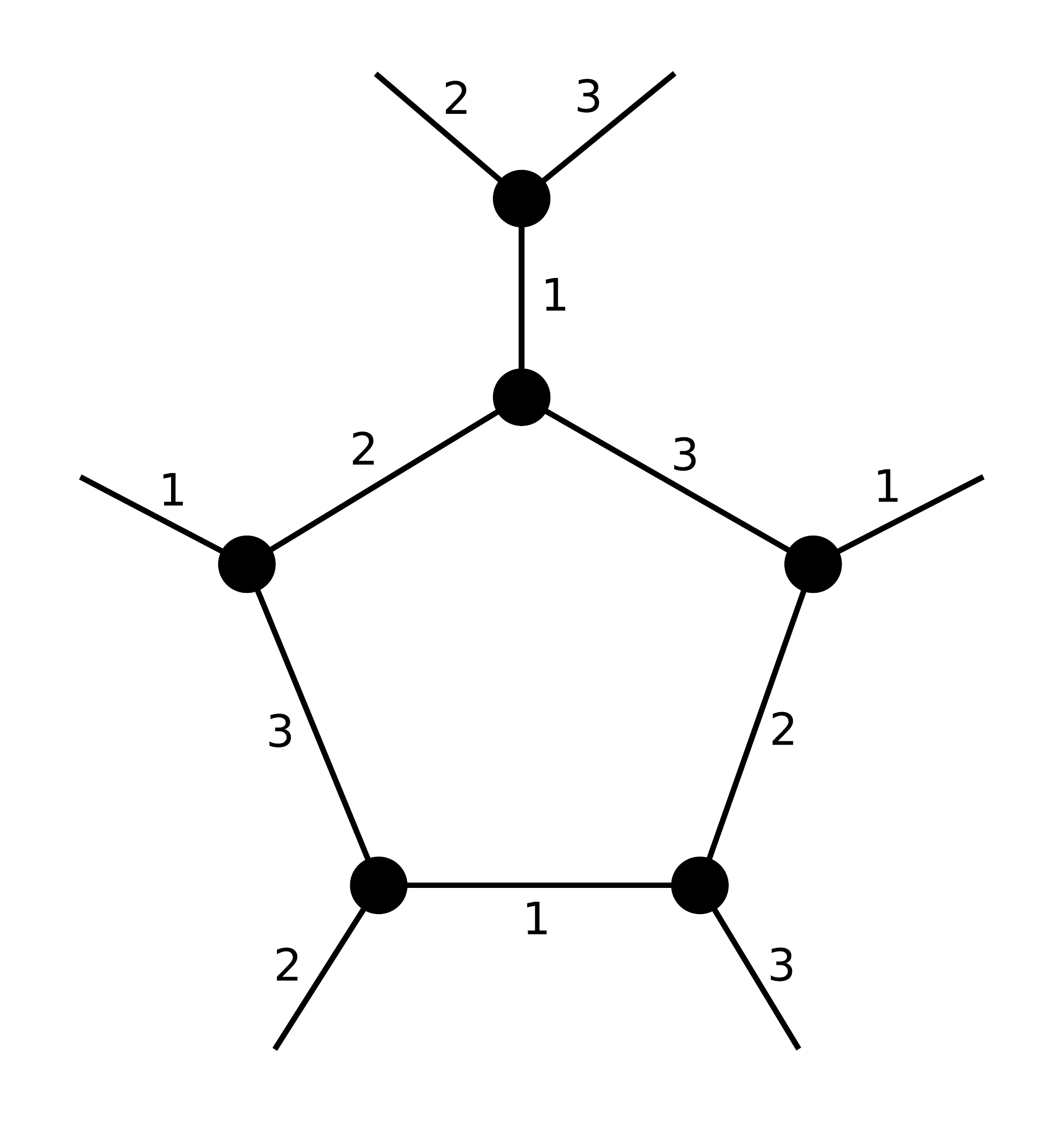}
\caption{A Tait coloring of a 6-pole}
\label{colored-6-pole}
\end{center}
\end{figure}

\noindent{\bf The Parity Lemma.}
{\em Let $M$ be a Tait colored $m$-pole with $m_i$ semiedges having color $i=1,2,3$. Then,
\begin{equation}
\label{parity}
m_1\equiv m_2 \equiv m_3 \equiv m\qquad (\mod\ 2).
\end{equation}
}

This result has been used extensively in the construction of snarks (see, for instance, Isaacs \cite{i75}, Goldberg \cite{go81}, Chetwynd and Wilson \cite{cw81}, Watkins \cite{w83}, Kochol \cite{k96} and Steffen \cite{es98}).

\section{Color completeness}

Given a multipole $M$ with semiedges $\epsilon_1,\ldots,\epsilon_m$, its set of semiedge colorings or {\em states} is
$$
\Col(M)=\{(\phi(\epsilon_1),\ldots,\phi(\epsilon_m) : \mbox{$\phi$ is a Tait coloring of $M$}\}.
$$
A multipole $M$ is {\em color complete} if it has all the possible semiedge colorings allowed by the Parity Lemma. We call these states \textit{admissible}. If a given multipole $M$ has a  state $S$, we say that $S$ is \textit{realizable} in $M$. Two states are \textit{equivalent} if they can be transformed into each other by a permutation of colors, and two multipoles are \textit{color equivalent} if they have the same set of states. Not surprisingly, the number of states $\sigma(m)=|\Col(M)|$ of a color complete $m$-pole coincides with the number of possible colorings of the planar regions in the $m$-ring of an almost triangulation \cite[p. 190]{rgf98}. This is because of the equivalence, proved by Tait \cite{t80}, that the regions of a planar cubic graph $G$ are 4-colorable if and only if $G$ is 3-edge-colorable. Here we use the properties of 3-edge-colored multipoles to give a short derivation of the formula for such a number.

\begin{lema}
\label{number-states-complete}
Let $M$ be a color complete $m$-pole. Then, its number of states is
$$
\sigma(m)=\frac{1}{8}[3^{m-1}+2+(-1)^m 3].
$$
\end{lema}
\begin{proof}
We prove that $\sigma(m)$ satisfies the following recurrence:
\begin{equation}
\label{recurC(m)}
\sigma(m)=2\sigma(m-1)+3\sigma(m-2)-1,
\end{equation}
with initial values $\sigma(2)=\sigma(3)=1$. Note that, given the semiedges $\epsilon_1$ and $\epsilon_2$, the set of states can be partitioned into two subsets: Those with $\phi(\epsilon_1)=\phi(\epsilon_2)$ and those with $\phi(\epsilon_1) \neq \phi(\epsilon_2)$ for some Tait coloring $\phi$. Here and henceforth, the letters $a,b,c$ stand for the colors 1,2,3 in any order, and `$*$' denotes any color. Then we have the following facts:
\begin{itemize}
\item[$(i)$]
Every real\-izable state $(a,*, \stackrel{(m-2)}{\ldots}, *)$ of an $(m-1)$-multipole
gives rise to two \mbox{realizable} states, $(b,c,*,  \stackrel{(m-2)}{\ldots}, *)$, and $(c,b,*, \stackrel{(m-2)}{\ldots}, *)$, of an $m$-pole.
\item[$(ii)$]
Every state $(a,*, \stackrel{(m-3)}{\ldots}, *)$ of an $(m-2)$-multipole
gives rise to three realizable states, $(a,a,a,*, \stackrel{(m-3)}{\ldots}, *)$, $(b,b,a,*, \stackrel{(m-3)}{\ldots}, *)$, and $(c,c,a,*, \stackrel{(m-3)}{\ldots}, *)$, of an $m$-pole.
\item[$(iii)$]
If $m$ is odd, in $(i)$ the realizable state $(a, \stackrel{(m-1)}{\ldots}, a)$ gives two equivalent states. \mbox{Similarly}, if $m$ is even,  in $(ii)$ the realizable state $(a,\stackrel{(m-2)}{\ldots}, a)$ gives two equivalent states (in both cases, up to permutation of colors $b$ and $c$).
\end{itemize}
Then, a particular solution of \eqref{recurC(m)} is $\sigma(m)=1/4$, whereas a  general solution of the corresponding homogeneous equation turns out to be $\sigma(m)=\alpha 3^{m}+\beta (-1)^m$ where $\alpha,\beta\in \R$. Adding up and imposing the initial conditions we obtain the result.
\end{proof}

\noindent\textbf{Observation}: If $m_1 \geq 2$ and $m_2 \geq 2$, a straightforward calculation gives the following inequality:
$$\sigma(m_1)·\sigma(m_2)<\sigma(m_1+m_2).$$
Note that $m_1=1$ or $m_2=1$ are trivial situations, because a 1-pole is not Tait colorable ($\sigma(1)=0$). As a consequence of the inequality, a color complete multipole must be connected. Otherwise, the total number of its states, as obtained from the combination of states of its connected components, would be too small.

Color complete multipoles can be constructed recursively by using the following result:
\begin{propo}
Let $M_1$ and $M_2$ be two color complete multipoles, with $m^{(1)}+r$ and $m^{(2)}+r$  semiedges respectively, where $m^{(1)}, m^{(2)}\ge 2$, $r\ge 2$ if $r$ is even, and $r\ge 3$ if $r$ is odd.  For $k=1,\ldots,r$, choose the semiedges $(u_k)$  from $M_1$ and $(v_k)$ from $M_2$ and join them to obtain the edges $(u_k,v_k)$. Then, the obtained multipole $M$, with $m=m^{(1)}+m^{(2)}$ semiedges, is also color complete.
\end{propo}

\begin{proof}
Given a state $S$ of $M$,
and for $j=1,2$ and $i=1,2,3$, let $m_i$, $r_i$ and $m^{(j)}_i$ be, respectively, the number of semiedges of $M$, the number of edges $(u_i,v_i)$ of $M$, and the number of semiedges of $M_j$, having color $i$. Then, $m^{(j)}=m^{(j)}_1+m^{(j)}_2+m^{(j)}_3$ and \mbox{$r=r_1+r_2+r_3$} and, by the Parity Lemma, we must have (all congruences are modulo 2, and $i=1,2,3$):
\begin{align}
m_i^{(1)}+m_i^{(2)} & \equiv m^{(1)}+ m^{(2)},  \label{comp1}\\
m_i^{(1)}+r_i & \equiv m^{(1)}+r,  \label{comp2}\\
m_i^{(2)}+r_i & \equiv m^{(2)}+r. \label{comp3}
\end{align}
Now, we claim that there exist states $S_1$ of $M_1$ and $S_2$ of $M_2$ inducing the state  $S$ of $M$.
In other words, given $m_i^{(1)}$ and $m_i^{(2)}$ satisfying \eqref{comp1}, we want to find some values of $r_i$, $i=1,2,3$, satisfying \eqref{comp2} and \eqref{comp3}. This gives
\begin{align}
r_i & \equiv m^{(1)}+r-m_i^{(1)},  \nonumber\\
r_i & \equiv m^{(2)}+r-m_i^{(2)}, \nonumber
\end{align}
which has a solution if and only if the two right hand terms have the same parity. That is,
$$
m^{(1)}-m_i^{(1)}\equiv m^{(2)}-m_i^{(2)}.
$$
But this is precisely the condition \eqref{comp1}.
To see why the condition on $r$ is necessary, note first that the case $r=0$ is discarded by the previous observation that a color complete multipole must be connected. Hence, if $r$ is even, then $r \geq 2$. Consider now the case $r=1$. We are assuming that $m^{(1)} \geq 2$. If $m^{(1)}+r$ is odd, the state with $m^{(1)}_1 \equiv m^{(1)}_2 \equiv m^{(1)}_3 \equiv \alpha$, where $\alpha \in \{0,1\}$, imposes $r_1 \equiv r_2 \equiv r_3 \equiv 1-\alpha$. Therefore, if $\alpha=0$ then $r_1,r_2,r_3 \geq 1$, and if $\alpha=1$ then $r_i \geq 2$ for some $i \in \{1,2,3\}$. The case with even $m^{(1)}+r$ is similar.
\end{proof}

As a consequence, we have the following result on the minimum order $n(m)$ of a color complete $m$-pole:

\begin{propo}
For a given $m\ge 5$, the minimum number of vertices of a color complete $m$-pole satisfies the bounds:
\begin{itemize}
\item[$(a)$]
If $m$ is odd, then $m+2\le n(m)\le 10m-37$.
\item[$(b)$]
If $m$ is even, then $m+2\le n(m)\le 10m-40$.
\end{itemize}
\end{propo}

\begin{proof}
First we deal with the lower bounds.
In the trivial case of the 3-pole $\vv$ with only one vertex, the Parity Lemma implies that the three semiedges  have different color. The one-vertex 3-pole, having a single admissible state, is trivially complete and the only complete multipole with two semiedges incident to the same vertex.

For $m \geq 4$, it is obvious that a complete $m$-pole cannot have two semiedges incident to the same vertex. This would preclude a common color for both semiedges, a possibility admitted by the Parity Lemma. By the observation above, a color complete multipole must be connected. Moreover, if a connected multipole with no two semiedges at distance 0 has $m$ semiedges and only $m$ vertices, then it must be a cycle with a semiedge on each vertex, but such an $m$-pole cannot be complete either (again, with the trivial exception of the cycle 3-pole) because certain states admissible by the Parity Lemma are not possible in a cycle $m$-pole (see Section 5, Figure 5). Also according to the Parity Lemma, a multipole with $m$ vertices cannot have $m+1$ semiedges. Summarizing, if $m \geq 4$, a complete $m$-pole must have at least $m+2$ vertices.

Now we deal with the upper bounds.
Color complete multipoles can be constructed recursively as seen in Proposition 3.2, and specifically with $r=2$. Explicit color complete multipoles are known for $m=5$, Fig. \ref{color complete 5-6-pole}$(a)$, and $m=6$, Fig. \ref{color complete 5-6-pole}$(b)$, with 13 and 20 vertices respectively. This allows us to start the recursive \mbox{construction} of color complete multipoles with higher values of $m$. Let $M_m$ denote a color complete $m$-pole of minimum order, and $|V(M_m)|$ its number of vertices.
Then, for $m \geq 7$,
$$
|V(M_m)|=|V(M_{(m-2)+6-4})| \leq |V(M_{m-2})|+|V(M_6)| = |V(M_{m-2})|+ 20,
$$
and this gives $|V(M_m)| \leq 10m-37$ for $m$ odd, and
$|V(M_m)| \leq 10m-40$ for $m$ even, as claimed. \end{proof}

Although we think that these bounds can be improved, we already see a linear \mbox{behavior}. As a corollary of the recursive construction, we conclude that color complete $m$-poles exist for any $m \geq 2$.\\

\begin{figure}[t]
\begin{center}
\includegraphics[width=15cm]{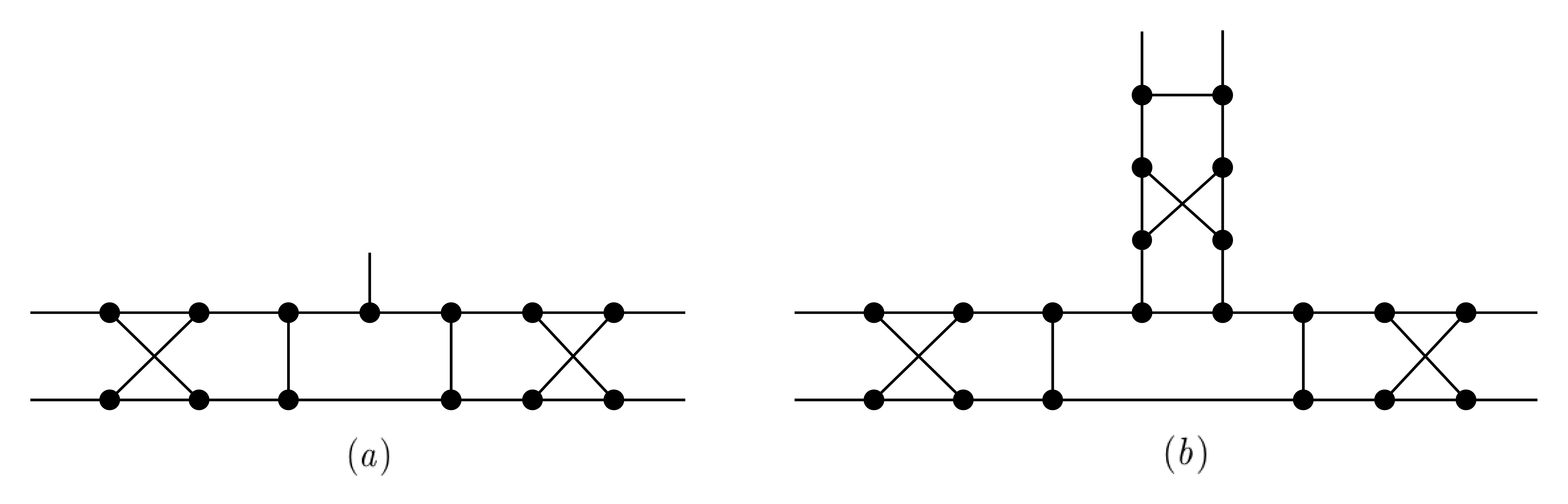}
\caption{Color complete 5- and 6-poles}
\label{color complete 5-6-pole}
\end{center}
\end{figure}

\section{Color closedness}

Following the terminology of Nedela and Skoviera \cite{ns96}, we say that an $m$-pole is \textit{color closed} if its state set has non empty intersection with the state set of any other Tait colorable \mbox{$m$-pole}. Obviously, such a multipole cannot be contained in a snark because there is always at least one possible coloring of its $m$ semiedges. A color complete multipole is trivially color closed.

If $\mu(m)$ and $\sigma(m)$ are respectively the minimum and maximum number of states of an $m$-pole, then, by the pigeonhole principle, any $m$-pole with more than $\sigma(m)-\mu(m)$ states must be color closed. If it has less than $\sigma(m)$ states, then it is non-trivially color closed.

\begin{lema}
\label{number-states-minimum}
The minimum number of states of Tait colorable $4$-, $5$-, and $6$-poles are 2, 3, and 5 respectively.
\end{lema}

\begin{proof}
Assume a 4-pole has state $S$. As a consequence of the Parity Lemma, we may assume without loss of generality that either $S=(a,a,a,a)$ or $S=(a,a,b,b)$. Start a Kempe interchange of $a$-$b$ type at the first semiedge. If $S=(a,a,a,a)$, the result must be either $(b,b,a,a)$, $(b,a,b,a)$ or $(b,a,a,b)$, none of them equivalent to $S$. If $S=(a,a,b,b)$, the result must be either $(b,b,b,b)$, $(b,a,a,b)$ or $(b,a,b,a)$, none of them equivalent to $S$. Then, a Tait colorable 4-pole has at least two states. The 4-pole consisting of 2 free edges realizes this possibility. Therefore, $\mu(4)=2$.

Assume a 5-pole has state S. Again as a consequence of the Parity Lemma, we may assume without loss of generality that  $S=(a,a,a,b,c)$. Start a Kempe interchange of $a$-$b$ type at the first semiedge. The result must be either $(b,b,a,b,c)$, equivalent to $(a,a,b,a,c)$, $(b,a,b,b,c)$, equivalent to $(a,b,a,a,c)$, or $(b,a,a,a,c)$. Alternatively, start a Kempe interchange of $a$-$c$ type at the first semiedge. The result must be either $(c,c,a,b,c)$, equivalent to $(a,a,b,c,a)$, $(c,a,c,b,c)$, equivalent to $(a,b,a,c,a)$, or $(b,a,a,c,a)$. Then, a Tait colorable 5-pole must have at least three non-equivalent states. The 5-pole consisting of the disjoint union of $\vv$ and one free edge $\e$ realizes this possibility. Therefore, $\mu(5)=3$.

Assume a 6-pole has state S. We may assume without loss of generality that either $S=(a,a,a,a,a,a)$, $S=(a,a,a,a,b,b)$, or $S=(a,a,b,b,c,c)$. First we prove that a 6-pole with the state $(a,a,a,a,a,a)$ has at least 5 states. Start a Kempe interchange of $a$-$b$ type at the first semiedge. We may assume that the resulting state is $(b,b,a,a,a,a)$. Alternatively, start a Kempe interchange of the same type at the third semiedge. Now we know that the Kempe chain cannot finish at the first or second semiedge. We may assume that the resulting state is $(a,a,b,b,a,a)$. Another possibility is to start the same type of Kempe interchange at the fifth semiedge. Now we know that it cannot finish at one of the first four semiedges. Hence, we assume that the resulting state is $(a,a,a,a,b,b)$. Summarizing, we can obtain at least three new states from $(a,a,a,a,a,a)$. Now take anyone of the three new states and start an arbitrary Kempe interchange of type $a$-$c$. This will give one new state, totalling five. The 6-pole consisting of 3 free edges realizes this possibility. Therefore, $\mu(6) \le 5$.

We must see that if a 6-pole does not have the state $(a,a,a,a,a,a)$ then it cannot have less than 5 states. First assume that the 6-pole has at least one state with two colors, say $S_1=(a,a,a,a,b,b)$. Starting a Kempe interchange of type $a$-$b$ at the last semiedge we obtain one of the states $(b,a,a,a,b,a)$, $(a,b,a,a,b,a)$, $(a,a,b,a,b,a)$ or $(a,a,a,b,b,a)$, none of them equivalent to $S$. Assume it is $S_2=(a,a,a,b,b,a)$. Using arbitrary Kempe interchanges of type $a$-$c$ we can obtain two new three-colored non-equivalent states, one from $S_1$ and one from $S_2$. This gives a total of four states $\{S_1,S_2,S_3,S_4\}$. But taking $S_1$ and starting a Kempe interchange of type $a$-$b$ at the first semiedge we obtain one of the states $(b,b,a,a,b,b)$, $(b,a,b,a,b,b)$, $(b,a,a,b,b,b)$, $(b,a,a,a,a,b)$ or $(b,a,a,a,b,a)$, the first three being equivalent to $(a,a,b,b,a,a)$, $(a,b,a,b,a,a)$ and $(a,b,b,a,a,a)$. None has an equivalent state in $\{S_1,S_2,S_3,S_4\}$. Therefore, a fifth non-equivalent state $S_5$ exists.

Now assume that the 6-pole has only three-colored states, and that one of these states is $S_1=(a,a,b,b,c,c)$. Using Kempe interchanges and reasoning as in the previous paragraphs, we may assume that the 6-pole has the states $S_2=(a,b,a,b,c,c)$, $S_3=(a,c,b,b,a,c)$ and $S_4=(a,a,b,c,b,c)$. All of them have repeated consecutive colors. But then a Kempe interchange of type $a$-$c$ starting at the first semiedge of $S_4$ gives either the state $(c,a,b,a,b,c)$ or $(c,a,b,c,b,a)$. None of them has repeated consecutive colors, and hence none of them has an equivalent state in $\{S_1,S_2,S_3,S_4\}$. Therefore, a fifth non-equivalent state $S_5$ exists. Summarizing, $\mu(6)=5$.
\end{proof}

The construction of color complete and color closed multipoles can be easily done by aggregation. In particular, the addition of the $X$-shaped 4-pole of Fig. \ref{color closed 4-pole}$(a)$ to an arbitrary \mbox{$m$-pole} by the junction of two pairs of semiedges is useful to increase its number of possible states (note the use of the $X$-shaped 4-pole in Section 3, Fig. \ref{color complete 5-6-pole}). In this way, we have obtained the color closed multipoles of Figs. \ref{color closed 4-pole}$(b)$ and \ref{color closed 5-6-pole}$(a),(b)$, with 4, 5 and 6 semiedges respectively, and 3, 8 and 27 states respectively. The $X$-shaped 4-pole is also color closed.

\begin{figure}[t]
\begin{center}
\includegraphics[width=12cm]{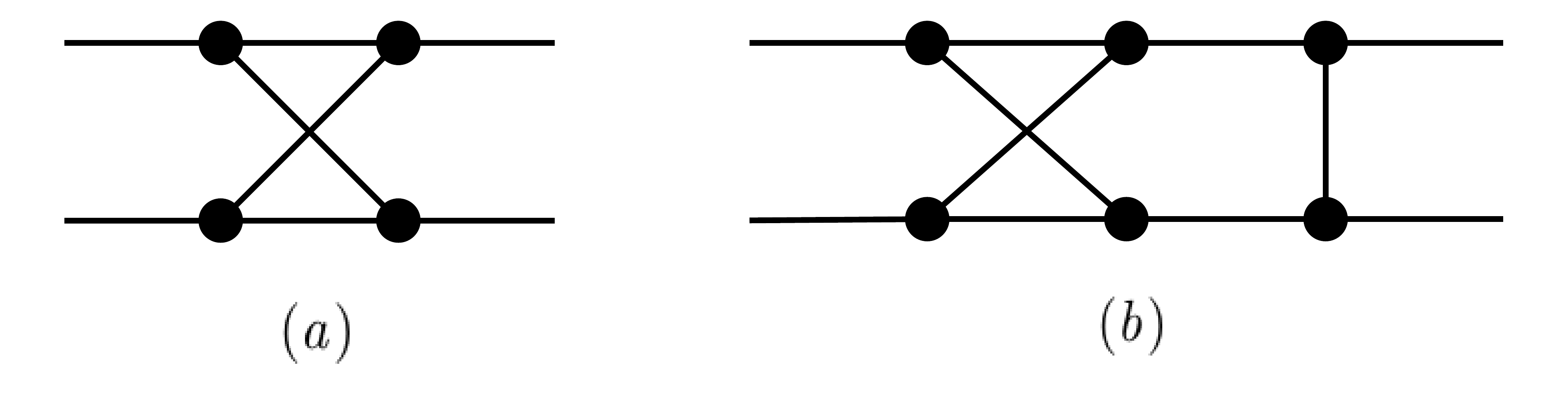}
\caption{The $X$-shaped 4-pole and another color closed 4-pole}
\label{color closed 4-pole}
\end{center}
\end{figure}

\begin{figure}[t]
\begin{center}
\includegraphics[width=15cm]{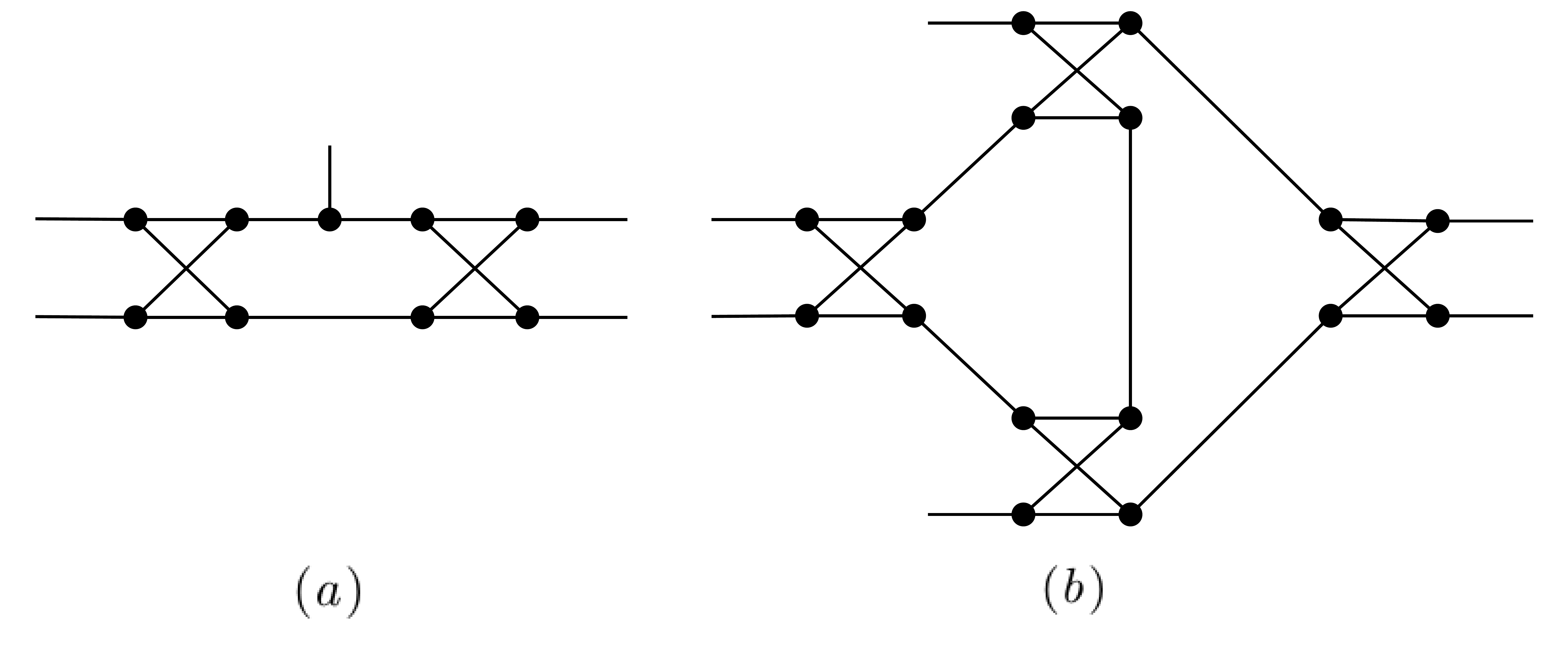}
\caption{Color closed 5- and 6-poles}
\label{color closed 5-6-pole}
\end{center}
\end{figure}

In fact, color closed multipoles need not have a number of states close to the maximum, because many state sets are not \textit{realizable} (that is, they are not realized by any multipole). In \cite{kmn13}, color closed 6-poles with only 15 and 16 states are described.

The \textit{minimal multipoles} are the $m$-poles with the minimum possible
number of vertices.
If $m$ is even, the minimal $m$-pole is constituted by $m/2$ free edges $\e$
and it has no vertices. If $m$ is odd, the minimal $m$-pole is formed by
the minimal 3-pole $\vv$, and $(m-3)/2$ free edges.

\begin{lema}
\label{num-states-minmultipol}
Let $\rho(m)$ be the number of states of a minimal multipole. Then,
\begin{itemize}
\item[$(a)$]
If $m$ is odd, then $\rho(m)=3^{\frac{m-3}{2}}$,
\item[$(b)$]
If $m$ is even, then $\rho(m)=\textstyle\frac{1}{2}(3^{\frac{m}{2}-1}+1)$.
\end{itemize}
\end{lema}

\begin{proof}
$(a)$ The minimal 3-pole $\vv$ has a unique coloring, and each one of the $(m-3)/2$ free edges has 3 possible colors.

$(b)$ We have the following recurrence relation: $\rho(m+2)=3\rho(m)-1$. To show this, consider $(m+2)/2$ free edges. If the first $m/2$ have the same color, say $a$, then two of the three possible color assignments for the last free edge, $b$ and $c$, yield equivalent states because the colors are interchangeable. \mbox{Otherwise}, all color assignments for the last free edge yield non-equivalent states. Therefore, only two of the $3\rho(m)$ possible states are equivalent to each other, and we must discard one of them. This gives the recurrence. The initial value is $\rho(2)=1$, corresponding to the unique state of a free edge. Using induction on $m$, or the same technique as in Section 3, we obtain the stated result.
\end{proof}

From Lemmas \ref{number-states-complete}, \ref{number-states-minimum}, and \ref{num-states-minmultipol}, we see that the first values of $\mu(m)$, $\rho(m)$, and $\sigma(m)$ are:
\begin{align*}
\mu(4)&=\rho(2)=2, & \sigma(4)&=4,\\
\mu(5)&=\rho(3)=3, & \sigma(5)&=10,\\
\mu(6)&=\rho(6)=5, & \sigma(6)& =31.
\end{align*}

From the results obtained, we conjecture that the minimum numbers of states are always 
attained by the minimal multipoles.

\section{Irreducibility}

For a given $m$-pole $M$, let $\Col(M)$ denote the set of states of $M$.
Then, $M$ is said to be \mbox{{\em reducible}} when there exists an $m$-pole $N$ such that $|V(N)|<|V(M)|$ and \mbox{$\Col(N)\subseteq \Col(M)$}. Otherwise, we say that $M$ is {\em irreducible}. Obviously, all minimal multipoles are irreducible.

The following result was first proved in \cite{f91}:

\begin{propo}
Let $U$ be a snark. Then, for any integer $m\ge 1$ there exists a positive integer-valued function $v(m)$ such that any $m$-pole $M$ contained in $U$ with $|M|> v(m)$ is either not Tait colorable or reducible.
\end{propo}

A reformulation of the same result was given later by Nedela and \v{S}koviera \cite[Th. 3.1]{ns96}. In the same paper, the authors also gave credit to \cite{f91} for the introduction of the function $v(m)$, and rewrote it as $\kappa(k)$.

The known values of $v(m)$ are $v(2)=0$, $v(3)=1$ (both are trivial results), $v(4)=2$ (Goldberg \cite{go81}), and $v(5)=5$ (Cameron, Chetwynd, and Watkins \cite{ccw87}). For $m=6$, Karab\'a\v{s}, M\'ac\v{a}jov\'a, and  Nedela \cite{kmn13} proved that $v(6)\ge 12$, although they implicitly attributed the definition of $v(m)$ and the above result to Nedela and \v{S}koviera \cite{ns96}. 

The study of the function $v(m)$ has relevant importance in the study of the structure of  snarks. More precisely, every value of $v(m)$ has a corresponding result for the decomposition of snarks. Moreover, according to the Jaeger-Swart Conjecture \cite{jsw80}, every snark contains a cycle-separating edge-cut of size at most six. As commented in \cite{kmn13}, if this were the case, then $v(6)$ would be the most interesting unknown value of $v(m)$.

The following basic lemma is simple but very useful in our study:

\begin{lema}
Let $M$ be an irreducible multipole. Then, every submultipole $N\subset M$ is also irreducible.
\end{lema}

\begin{proof}
We proceed by contradiction. If a submultipole $N$ of $M$ is reducible to, say, a multipole $N'$, then $M$ is reducible to a multipole $M'$ obtained replacing $N$ by $N'$. This contradicts the fact that $M$ is irreducible.
\end{proof}

As an example of application, note that, as every $m$-pole with $n\ge 1$ vertices
contains an $(m-1)$-pole with $n-1$ vertices (just delete a semiedge $\epsilon=(v)$ and its incident vertex $v$ to create two new semiedges $\epsilon'$ and $\epsilon''$), we have that $v(m)\ge v(m-1)-1$. Hence, from the results in \cite{kmn13}, we know that $v(7)\ge 11$.

Note that $v(m)$ must have the same parity as $m$. We can be more precise with respect to the behavior of $v(m)$ if we use the following result.

\begin{lema}
Let $M$ be an irreducible multipole. Then, the disjoint union of $M$ and one free edge $\e$ is also irreducible.
\end{lema}

\begin{proof}
We proceed by contradiction. Assume that the disjoint union of $M$ and $\e$ is \mbox{reducible} to a submultipole $M'$. We denote the semiedges of $\e$ by $\epsilon_0$ and $\epsilon_1$, and their corresponding semiedges in $M'$ by $\epsilon'_0$ and $\epsilon'_1$. In the set of states of $M\cup\e$ (respectively $M'$), all states must have $\phi(\epsilon_0)=\phi(\epsilon_1)$ (respectively $\phi(\epsilon'_0)=\phi(\epsilon'_1)$). This implies that $\epsilon'_0$ and $\epsilon'_1$ are not adjacent (that is, $\dist(\ep'_0,\ep'_1)>0$) unless they constitute a free edge $\e'$ also in the submultipole $M'$, but then $M$ would be trivially reducible to $M'$ minus $\e'$ and we would be done. Otherwise, we may join $\epsilon_0$ to $\epsilon_1$ in $M\cup\e$ and $\epsilon'_0$ to $\epsilon'_1$ in $M'$. In the first case, we obtain $M$ and a trivial `closed edge' which can be ignored. In the second case, we obtain a multipole $M''$ whose number of vertices is $|V(M'')|=|V(M')|<|V(M+\e)|=|V(M)|$ and whose set of states is a subset of the set of states of $M$. By definition, $M$ would be reducible to $M''$, contradicting our assumption.
\end{proof}

As a consequence, we have that $v(m) \ge v(m-2)$. Therefore, the function $v(m)$ is, at least, partially monotone in the sense that its even (respectively odd) values behave monotonically.

\section{Trees and forests}

The number of states of a tree $m$-pole $T_m$, denoted by $t(m)=|\Col(T_m)|$, is easy to compute. Tree multipoles can be constructed recursively by the addition of vertices. A 3-pole with a single vertex has only one possible state. Then, $t(3)=|\Col(T_3)|=1$. Every new vertex increases $m$ one unit and doubles the number of possible states. Then, $t(m)=2t(m-1)$. As a consequence, $t(m)=2^{m-3}$.

An $m$-pole $M$ with $c$ components is called {\em separable} if there exists some  $m$-pole $N$, with at least $c+1$ components,
$|V(N)|=|V(M)|$, and $\Col(N)\subseteq  \Col(M)$.

Now, we have the following result.

\begin{propo}
Every tree multipole $T_m$ is non-separable and irreducible.
\end{propo}
\begin{proof}
The proof is by induction. The result is trivial for $m=2,3$ since $T_2=\e$ and $T_3=\vv$.
Then, assume that the results holds for some $T_m$ with $m>3$. For some semiedge $(u)$, consider the $(m+1)$-pole $M$ obtained by joining one semiedge of $\vv$ to $u$, so `turning' $(u)$ into two new semiedges $\epsilon_1=(v)_1$ and $\epsilon_2=(v)_2$. We first prove that $M$ is irreducible. On the contrary, there would exist an $(m+1)$-pole $M'$ with $|V(M')|< |V(M)|=m-1$ (and hence $|V(M')|\le m-3$) and $\Col(M')\subseteq  \Col(M)$.
Thus, $M'$ must have at least two connected components, say, $M'_1$ and $M'_2$. Moreover, since $\epsilon_1$ and $\epsilon_2$ must have different colors in any Tait coloring of
$M$, the same is true for the corresponding semiedges $\epsilon'_1$ and $\epsilon'_2$ in $M'$. Consequently, such semiedges cannot be in different components of $M'$. Without loss of generality, assume that $\epsilon'_1,\epsilon'_2\in {\cal E}(M'_1)$. Now by joining $\epsilon'_1$ and $\epsilon'_2$ to two semiedges of $\vv$, and doing the same operation to $\epsilon_1$ and $\epsilon_2$, so creating a digon, we respectively obtain the multipoles $N$ and $N'$ with the following properties:
\begin{itemize}
\item $\Col(N') \subseteq \Col(N)=\Col(T_m)$ (as the created digon 2-pole in $N$ is color equivalent to $\e$).
\item $|V(N')|\le m-2=|V(T_m)|$.
\end{itemize}
Thus, if $|V(N')|<m-2$, $T_m$ would be reducible, and if $|V(N')|=m-2$, $T_m$ would be separable.
In any case, we get a contradiction.

To prove that $M$ is also non-separable we again assume the contrary and iterate the above procedure until $M$ is proved to be color equivalent to some small tree multipole $T$, whereas $M'$ is color equivalent to a multipole $N$ constituted by some isolated edges. But this again contradicts the fact that $\Col(M')\subseteq  \Col(M)$.
\end{proof}

As every tree $m$-pole $T_m$ has $m-2$ vertices, the above proposition allows us to state that $v(m)\ge m-2$ for $m \ge 2$.

Using basically the same proof as above, we have the following:

\begin{propo}
\label{forest-irreducible}
Every forest multipole is non-separable and irreducible.
\end{propo}

The following result shows that the number of states of a forest $m$-pole depends only on $m$ and its number $n$ of vertices.

\begin{lema}
Let $f(n,m)$ denote the number of states of a forest $m$-pole $M$ with $n \leq m-2$. 
Then, if $n \leq m-4$, we have the following recurrence relation:
$$
f(n,m)=f(n,m-2)+f(n+1,m-1).
$$
\end{lema}

\begin{proof}
First notice that, if $n \leq m-4$, a forest $m$-pole with $n$ vertices has at least two components. 
Then, 
given two semiedges $\epsilon_1$ and $\epsilon_2$ (not in the same component), its set of states can be partitioned into two subsets: Those with $\phi(\epsilon_1)=\phi(\epsilon_2)$, and those with $\phi(\epsilon_1) \neq \phi(\epsilon_2)$, for some Tait coloring $\phi$ of $M$.

Case $\phi(\epsilon_1)=\phi(\epsilon_2)$: The two semiedges can be joined and we obtain an $m'$-pole $M'$ with $m'=m-2$ semiedges and $n'=n$ vertices. As a consequence of $n \leq m-4$, the relation $n' \leq m'-2$ holds. Therefore, its number of states is $f(n',m')=f(n,m-2)$. 

Case $\phi(\epsilon_1) \neq \phi(\epsilon_2)$: The two semiedges can be joined to two semiedges $\epsilon''_1$ and $\epsilon''_2$ of a minimal 3-pole $\vv$. Let $\epsilon''_3$ denote the remaining semiedge of $\vv$. The result of the junction is a new $m''$-pole $M''$ with $m''=m-1$ semiedges, now including $\epsilon''_3$, and $n''=n+1$ vertices. Again, the relation $n'' \leq m''-2$ holds as a consequence of $n \leq m-4$. Therefore, the number of states of $M''$ is $f(n'',m'')=f(n+1,m-1)$. 

\noindent Finally, adding up all states we obtain the recurrence relation.
\end{proof}

By means of a straightforward calculation, it can be proved that, if $n \leq m - 2(k+1)$, then the following formula holds:
\begin{equation}
\label{recurrence-f}
f(n,m)=\sum_{i=0}^k f(n+i,m-2k+i) {k \choose i}.
\end{equation}

\begin{lema}
\label{number-states-minimal}
\begin{itemize}
\item[$(a)$]
 The number of states of a $(2k+2)$-pole with no vertices is
$$
f(0,2k+2)=\frac{1}{2}(3^k+1).
$$
\item[$(b)$] The number of states of a $(2k+3)$-pole with one vertex is
$$
f(1,2k+3)=3^k.
$$
\end{itemize}
\end{lema}

\begin{proof}
Note that if $n=m-2>0$, then $f(n,m)=2^{n-1}$ because the corresponding $m$-pole is a tree. If $n=0$, then $f(0,2)=1$. 
Using (\ref{recurrence-f}):
\vskip 8pt
($a$) In the first case, we have:
\begin{align*}
f(0,2k+2)= &\sum_{i=0}^k f(i,i+2) {k \choose i} = f(0,2) + \sum_{i=1}^k f(i,i+2) {k \choose i}\\
= & 1 + \sum_{i=1}^k 2^{i-1} {k \choose i} = 1 + \frac{1}{2} \sum_{i=1}^k 2^i {k \choose i} = 1 + \frac{1}{2}(3^k-1) = \frac{3^k+1}{2}.
\end{align*}
\vskip 8pt
($b$) In the second case, we have:
$$f(1,2k+3)=\sum_{i=0}^k f(i+1,i+3) {k \choose i} = \sum_{i=0}^k 2^i {k \choose i} = 3^k.$$
\end{proof}

The following result, consequence of Lemma \ref{number-states-minimal}, implies that the minimum number of states of a forest $m$-pole $M$ is attained when $M$ is a minimal $m$-pole.

\begin{propo}
For a fixed even value of $m$, the minimum of $f(n,m)$ is $f(0,m)$, and for a fixed odd value of $m$, the minimum of $f(n,m)$ is $f(1,m)$.
\end{propo}

\begin{proof}
If $m$ is even, then $f(0,m)=\frac{1}{2}(3^{\frac{m-2}{2}}+1)$. If $m$ is odd, then $f(1,m)=3^{\frac{m-3}{2}}$. In both cases we obtain $\rho(m)$. To see that this is the minimum, note that recurrence (\ref{recurrence-f}) implies $f(n,m)=2^{n-1}\cdot 3^{\frac{m-n}{2}-1}$ for $1 \le n \le m-2$. As a consequence, if $n_1 \le n_2 \le m-2$, then $f(n_1,m) \le f(n_2,m)$. 
\end{proof}

\begin{figure}[h]
\begin{center}
\includegraphics[width=11cm]{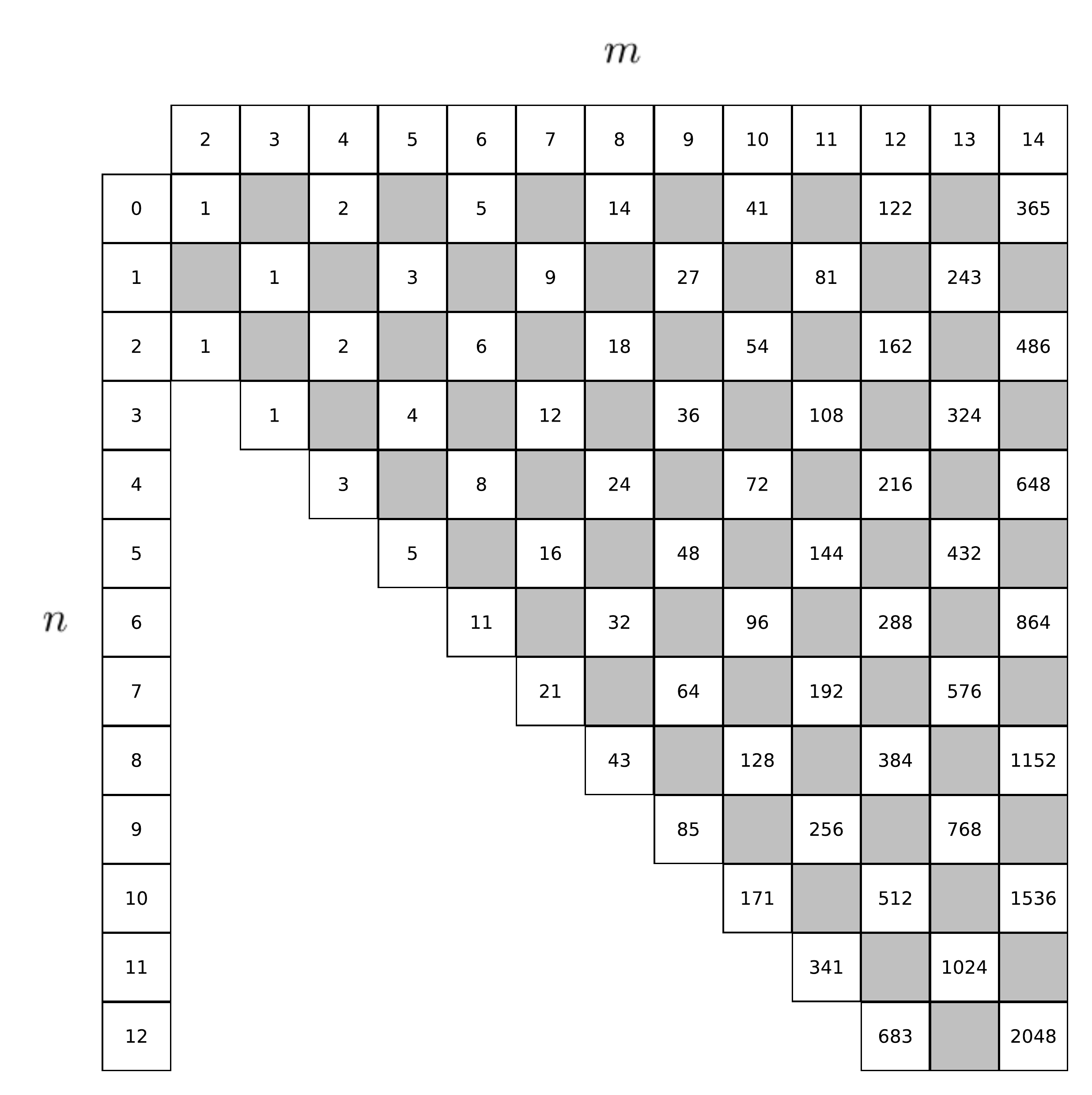}
\caption{Numbers of states of forest $m$-poles with $n$ vertices. The rows $n=0$ and $n=1$ correspond to minimal $m$-poles, and the diagonal $n=m-2$ corresponds to tree $m$-poles. The numbers of states of cycle $m$-poles, corresponding to the diagonal $n=m$, are also represented.}
\label{forest-multipoles-table}
\end{center}
\end{figure}

\section{Cycles}

In this section we prove the irreducibility of every cycle multipole for $m \ge 5$, and calculate the number of states of an arbitrary cycle multipole with $m \ge 1$.

\begin{propo}
The number $c(m)$ of states of a cycle $m$-pole $C_m$ is
$$
\textstyle c(m)=\frac{1}{3}(2^{m-1}+(-1)^m).
$$
\end{propo}
\begin{proof}
The coloring of the semiedges is determined by the coloring of the edges of the cycle. Then, the number $c(m)$ verifies the following recurrence:
$$
c(m)=c(m-1)+2c(m-2).
$$
To show this, consider one arbitrary edge of the cycle, and the two edges adjacent to it. If these two edges have different color, the color of the middle edge is forced. If they have the same color, there are two possibilities for the color of the middle edge. In the first case, we delete the middle edge and make its two adjacent edges adjacent to each other: the resulting cycle has $m-1$ edges and $c(m-1)$ 3-edge-colorings. In the second case, we delete the middle edge and identify its two adjacent edges as if they were one: the resulting cycle has $m-2$ edges and $c(m-2)$ 3-edge-colorings, but in this case the deleted edge had two possible colors. This gives the recurrence for $c(m)$. Its initial values are $c(1)=0$ and $c(2)=6$. Using induction on $m$, or the same technique as in Section 3, we can prove that $c(m)=2^m+2(-1)^m$. However, in a state of a multipole the 3 colors are interchangeable: we must divide by $3!$ to obtain the final result.
\end{proof}

Now we show that, as the tree multipoles, cycle multipoles are also irreducible. To this end, we need some previous lemmas.

\begin{lema}
\label{lema-c1}
Let $\ep_1,\ep_2,\ep_3$ be three $(ordered)$ semiedges of a $m$-tree multipole $T_m$ Then, the state $(a,b,a)$ is realizable unless
one of the following conditions hold:
\begin{itemize}
\item[$(i)$]
$\dist(\ep_1,\ep_3)=0$.
\item[$(ii)$]
$\dist(\ep_1,\ep_2)=\dist(\ep_2,\ep_3)=1$.
\end{itemize}
\end{lema}
\begin{proof}
Let $\ep_i=(v_i)$, $i=1,2,3$.
First note that, as $\ep_1$ and $\ep_2$ are distinct, a coloring of the form $(a,b,*)$ is always possible.
Then, the only cases where color $*$ cannot be $a$ are
$(i)$ (trivially since the $v_1=v_2$) and $(ii)$ since, then, edge $v_1v_2$ must have color $c$ and, hence, edge $v_2v_3$ must have color $a$ which is forbidden.
\end{proof}

\begin{figure}[t]
\begin{center}
\vskip-1cm
\includegraphics[width=7cm]{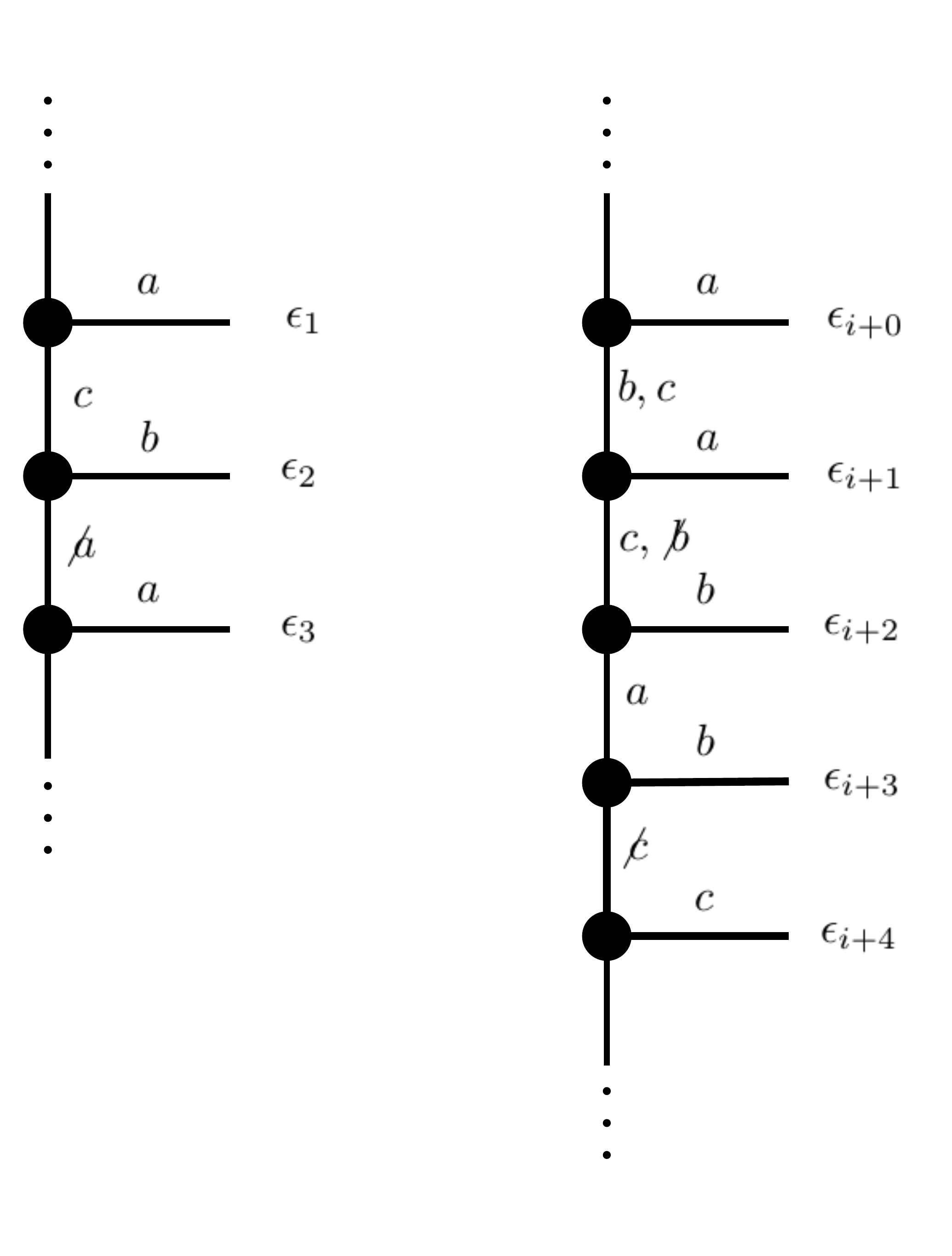}
\caption{Non-realizable states}
\label{forbiden-patterns}
\end{center}
\end{figure}

\begin{lema}
\label{lema-c2}
Let $\ep_i,\ep_{i+1},\ldots,\ep_{i+r-1}$, be $r$ successive semiedges of a $m$-cycle multipole $C_m$ (subindexes understood modulo $m$). Then, the following states are not realizable:
\begin{itemize}
\item[$(i)$]
For $r=3$, $(a,b,a)$.
\item[$(i)$]
For $r=5$, $(a,a,b,b,c)$.
\end{itemize}
\end{lema}
\begin{proof}
Let $\ep_{i+j}=(v_{i+j})$, $j=0,\ldots,4$.
Case $(i)$ is proved as in Lemma \ref{lema-c1}.
Case $(ii)$ follows also easily by considering the two different possible colors, $b$ and $c$, of edge $v_iv_{i+1}$ and concluding that both lead to a contradiction.
\end{proof}

\begin{figure}[h]
\begin{center}
\includegraphics[width=15cm]{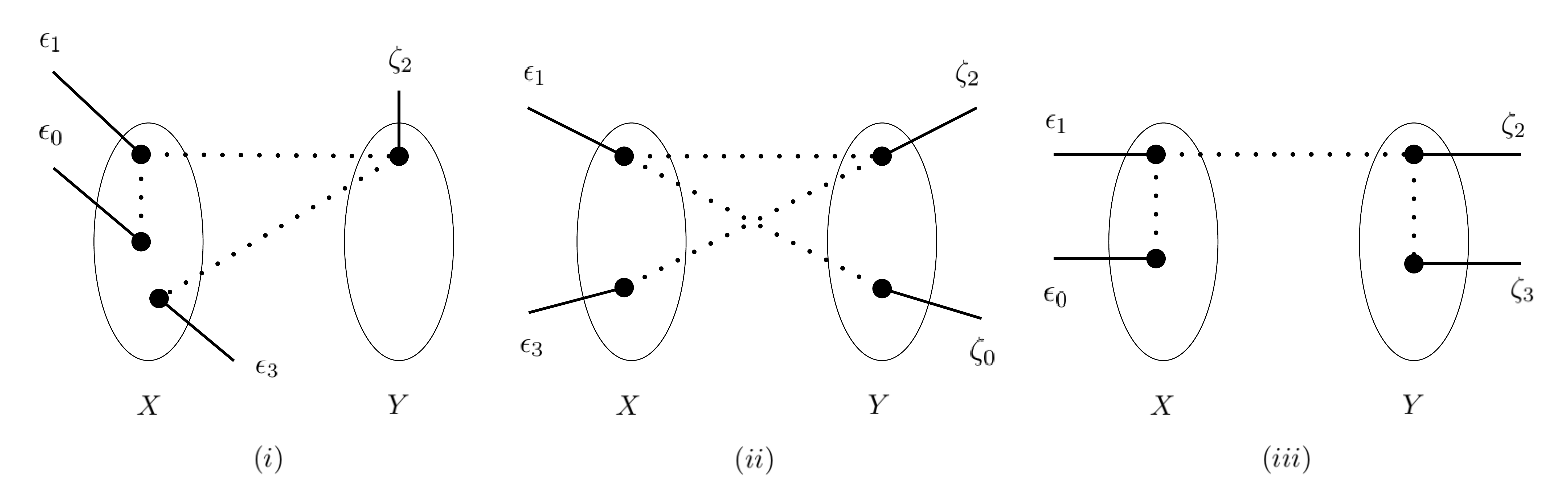}
\caption{Possible successions of semiedges}
\label{success-semiedges}
\end{center}
\end{figure}

\begin{propo}
Every cycle multipole $C_m$ with $m\ge 5$ is irreducible.
\end{propo}
\begin{proof}
The proof is by contradiction. Assume that $C_m$ is reducible to an $m$-pole $T$. As $|T|\le m-2$, $T$ is either a tree multipole or it has at least two components. Since $\Col(T)\subseteq \Col(C_m)$, there must be a cyclic ordering of the semiedges of $T$, $(\ep_0,\ldots,\ep_{m-1},\ep_0,\ldots)$, corresponding to the successive semiedges of $C_m$, such that every state of $T$ is also realizable in $C_m$.
Hovewer, we will prove that, no matter how we choose such an ordering, this is not the case. We consider three cases:
\begin{itemize}
\item[$(a)$] $T$ is a tree.
\begin{itemize}
\item[$(a1)$]
If $m$ is even, it could be that all the semiedges of $T$ come into pairs (that is, every pair of them is incident with a given vertex). Without loss of generality, suppose that $\ep_i$ and $\ep_{i+1}$ are incident to $u_i$ for $i=0,2,\ldots,(m-2)/2$. Then, by Lemma \ref{lema-c1}, in the cyclic ordering and in order to avoid the states $(\ldots,a,b,a,\ldots)$ we are forced to take $(\ldots,\ep_1,\ep_j,\ep_2,\ep_{j+1},\ldots)$ with $j \ge 3$. But every choice for the semiedge $\ep_k$ next to $\ep_{j+1}$ allows the state $(a,b,a)$ for the semiedges  $(\ep_2,\ep_{j+1},\ep_k)$, which is not realizable in $C_m$ by Lemma \ref{lema-c2}.
\item[$(a2)$]
If $m$ is odd, $T$ must have some semiedge $\zeta$ at nonzero distance from every other semiedge, and also some pair $(\ep_i,\ep_{i+1})$ of incident semiedges. Thus, at some point of the cyclic ordering, we must have $(\ldots,\zeta,\ep_j,\ep_k,\ldots)$ where $j\in\{i,i+1\}$. Then, again, all possible choices of $\ep_k$ allow the state $(a,b,a)$.

\end{itemize}

\item[$(b)$] $T$ has two components $X,Y$: Let $X$ and $Y$ have semiedges $\ep_i$ and $\zeta_j$, respectively. Notice that either $X$ or $Y$ has at least three semiedges. Then, at some point of the cyclic ordering, we must go from, say, $\ep_1\in \Ei(X)$ to $\zeta_2\in \Ei(Y)$. Up to symmetries, there are three possibilities for the next semiedges to $\ep_1$ and $\zeta_2$ (see Fig. \ref{success-semiedges} where `cyclic adjacency' between semiedges is represented by dotted lines):

\begin{itemize}

\item[$(b1)$] Case $(\ldots,\ep_0,\ep_1,\zeta_2,\ep_3,\ldots)$: Here, for any Tait coloring $\phi$ of $T$, and in order to avoid the state $(a,b,a)$ for the semiedge triples $(\ep_0,\ep_1,\zeta_2)$ and
        $(\ep_1,\zeta_2,\ep_3)$, it must be $\phi(\ep_0)=\phi(\ep_1)$ and $\phi(\ep_1)\neq \phi(\ep_3)$. Thus, the state of the successive semiedges must be $(a,a,*,b)$. Now consider the $a$-$b$ Kempe chain in $X$, starting from $\ep_3$ and interchange its colors to get the new Tait coloring $\phi'$. If it ends to a semiedge different from $\ep_0,\ep_1$ we get the new state $(a,a,*,a)$. Otherwise, if it ends to, say, $\ep_0$ we get $(b,a,*,b)$. In both cases we can adequately chose the color $*$ to get the state $(a,b,a)$ for three successive semiedges.

\item[$(b2)$] Case $(\ldots,\zeta_0,\ep_1,\zeta_2,\ep_3,\ldots)$: To avoid the state $(a,b,a)$ for the two semiedge triples it must be $\phi(\zeta_0)\neq\phi(\zeta_2)$ and  $\phi(\ep_1)\neq \phi(\ep_3)$. Thus, the state of the \mbox{successive} semiedges can be chosen to be $(a,a,b,b)$. Let us consider the \mbox{different} possibilities for the semiedge following $\ep_3$. If this is $\ep_4\in \Ei(X)$, the semiedges $(\ep_1,\zeta_2,\ep_3,\ep_4)$ follow the pattern as in $(i)$, which has already been dealt with. Alternatively, if such a semiedge is $\zeta_4\in \Ei(Y)$, it cannot have color $c$ because of Lemma \ref{lema-c2}$(ii)$. Thus, suppose that $\zeta_4$ has a color different from $c$, say $a$, and consider the $a$-$b$ Kempe chain in $Y$ starting from  $\zeta_4$ and interchange its colors. Then, reasoning as above, we obtain the realizable state $(a,b,a)$ for three successive semiedges.

\item[$(b3)$] Case $(\ldots,\ep_0,\ep_1,\zeta_2,\zeta_3,\ldots)$: Now, to avoid the state $(a,b,a)$ for the two semiedge triples it must be $\phi(\ep_0)=\phi(\ep_1)$ and  $\phi(\zeta_2)=\phi(\zeta_3)$, so that we can chose the state of the successive semiedges to be $(a,a,b,b)$. Thus, the semiedge following $\zeta_3$ cannot have color $c$ because of Lemma  \ref{lema-c2}$(ii)$. For the situation of such a semiedge, we consider again two possibilities: If $\zeta_4\in Y$
            and it has color $a$, the interchanging of color in the $a$-$c$ Kempe chain leads again to the state $(a,a,b,b,c)$. Otherwise, if $\zeta_4$ has color $b$, the $b$-$c$ Kempe chain leads again to either such a state or to the color $c$ for some of the semiedges $\zeta_2$ or $\zeta_3$. But, again, both alternatives led to a realizable state $(b,c,b)$ (changing the color of $\ep_1$ from $a$ to $b$, if necessary). Finally,
            if $\ep_4\in X$ and it has colors $a$ or $b$, we reason as above with the respective $a$-$c$ or $b$-$c$ Kempe chains to reach the
            states of type $(a,a,b,b,c)$ or $(a,b,a)$.
\end{itemize}

\item[$(c)$] $T$ has at least three components $X,Y,Z$:
Denote the semiedges of $X$ and $Y$ as above and let $Z$ have semiedges $\eta_k$.
Note that in the cyclic ordering, and to avoid the state $(a,b,a)$, we cannot have 3 successive semiedges in different components $X,Y,Z$.
Then, the only possibility to be considered, being different from those in $(b)$,  is $(\ep_0,\zeta_1,\zeta_2,\eta_3)$. Then it must be $\phi(\zeta_1)=\phi(\zeta_2)$ and the corresponding state can be chosen to be $(c,b,b,a)$. Now let us check the different possibilities for the other semiedge next to $\eta_3$. First, from the above comment, it cannot be in $X$. Second it cannot have color $a$  (Lemma \ref{lema-c2}$(ii)$) and, hence, it cannot be in $Z$. Thus, such a semiedge should be $\zeta_4\in Y$, and should have color $b$ or $c$. But, reasoning as above, the consideration of the $b$-$a$ or $c$-$a$, respectively, Kempe chains starting in $\zeta_4$ leads to the non-desired states of Lemma \ref{lema-c2}.
\end{itemize}
Summarizing, $T$ has always a state not realizable in $C_m$. But this is in contradiction with our assumption. We conclude that $C_m$ is irreducible.
\end{proof}

As every cycle $m$-pole $C_m$ has $m$ vertices, the above proposition allows us to state that if $m \ge 5$ then $v(m)\ge m$.

\section{Conclusions}

We have calculated the exact number of states of tree, cycle, minimal and color complete multipoles with $m$ semiedges, and proved that color complete multipoles exist for every non trivial value of $m$. The minimum order of a color complete multipole has lower and upper linear bounds in terms of $m$. Moreover, we have seen that results on the behavior of the function $v(m)$, which is barely known but relevant in the study of the structure of snarks, can be obtained by the analysis of reducibility of multipoles, in particular tree and cycle multipoles.

\end{document}